\documentclass[12pt]{amsart}
\usepackage{amsfonts,amssymb,amscd,amsmath,enumerate,verbatim,calc}
\usepackage{graphicx}

\newtheorem{Theorem}{Theorem}[section]
\newtheorem{Lemma}[Theorem]{Lemma}
\newtheorem{Corollary}[Theorem]{Corollary}
\newtheorem{Proposition}[Theorem]{Proposition}

\begin{document}
\title{Locating-Dominating Sets of Functigraphs}
\author{Muhammad Murtaza, Muhammad Fazil, Imran Javaid$^{*}$, Hira Benish}
\keywords{Location-domination number, Functigraph.
$^*$ Corresponding author: imran.javaid@bzu.edu.pk}

\address{Centre for advanced studies in Pure and Applied Mathematics,
Bahauddin Zakariya University Multan, Pakistan
\newline Email: mahru830@gmail.com, mfazil@bzu.edu.pk, imran.javaid@bzu.edu.pk,
\newline hira\_benish@yahoo.com}

\date{}
\maketitle
\begin{abstract} A locating-dominating set of a graph $G$ is a dominating set of $G$ such that every vertex of
$G$ outside the dominating set is uniquely identified by its
neighborhood within the dominating set. The location-domination
number of $G$ is the minimum cardinality of a locating-dominating
set in $G$. Let $G_{1}$ and $G_{2}$ be the disjoint copies of a
graph $G$ and $f:V(G_{1})\rightarrow V(G_{2})$ be a function. A
functigraph $F^f_{G}$ consists of the vertex set $V(G_{1})\cup
V(G_{2})$ and the edge set $E(G_{1})\cup E(G_{2})\cup
\{uv:v=f(u)\}$. In this paper, we study the variation of the
location-domination number in passing from $G$ to $F^f_{G}$ and find
its sharp lower and upper bounds. We also study the
location-domination number of functigraphs of the complete graphs
for all possible definitions of the function $f$. We also obtain the
location-domination number of functigraph of a family of spanning
subgraph of the complete graphs.
\end{abstract}

\section{Introduction}
Locating-dominating sets were introduced by Slater \cite{slater1,
slater3}. The initial application of locating-dominating sets was
fault-diagnosis in the maintenance of multiprocessor systems
\cite{Karpov}. The purpose of fault detection is to test the system
and locate the faulty processors. Locating-dominating sets have
since been extended and applied. The decision problem for
locating-dominating sets for directed graphs has been shown to be an
NP-complete problem \cite{char2}. A considerable literature has been
developed in this field (see
\cite{ber,char1,col,fin,hon,rall,slater1,slater2}). In \cite{cac1},
it was pointed out that each locating-dominating set is both
locating and dominating set. However, a set that is both locating
and dominating is not necessarily a locating-dominating set.

We use $G$ to denote a connected graph with the vertex set $V(G)$
and the edge set $E(G)$. The \emph{degree} of a vertex $v$ in $G$,
denoted by $deg(v)$, is the number of edges to which $v$ belongs.
The \emph{open neighborhood} of a vertex $u$ of $G$ is $N(u)=\{v\in
V(G):uv\in E(G)\}$ and the \emph{closed neighborhood} of $u$ is
$N[u]=N(u)\cup \{u\}$. Two vertices $u,v$ are \emph{adjacent twins}
if $N[u]=N[v]$ and \emph{non-adjacent twins} if $N(u)=N(v)$. If
$u,v$ are adjacent or non-adjacent twins, then $u,v$ are
\emph{twins}. A set of vertices is called a \emph{twin-set} if every
two distinct vertices of the set are twins.

 Formally, we define a locating-dominating set as: A subset $L_D$ of
 the vertices of a graph $G$ is
called a locating-dominating set of $G$ if for every two distinct
vertices $u,v \in V(G)\setminus L_D$, we have $\emptyset \neq
N(u)\cap L_D\neq N(v)\cap L_D\neq\emptyset$. The {\it
location-domination number}, denoted by $\lambda(G)$, is the minimum
cardinality of a locating-dominating set of $G$.

The functigraph has its foundations back in the idea of permutation
graph \cite{char} and mapping graph \cite{dor}. A permutation graph
of a graph $G$ with $n$ vertices consists of two disjoint identical
copies of $G$ along with $n$ additional edges between the two copies
according to a given permutation on $n$ points. In a mapping graph,
the additional $n$ edges between the two copies are defined
according to a given function between the vertices of the two
copies. The mapping graph was rediscovered and studied by Chen et
al. \cite{yi}, where it was called the functigraph. Thus, a
functigraph is the generalized form of permutation graph in which
the function $f$ need not necessarily a permutation. In the recent
past, a number of graph variants were studied for functigraphs. Eroh
et al. \cite{Linda2} studied that how metric dimension behaves in
passing from a graph to its functigraph and investigated the metric
dimension of functigraphs on complete graphs and on cycles. Eroh et
al. \cite{Linda1} investigated the domination number of functigraph
of cycles in great detail, the functions which achieve the upper and
lower bounds. Qi et al. \cite{GuQu,QiWang} investigated the bounds
of chromatic number of functigraph. Kang et al. \cite{kang}
investigated the zero forcing number of functigraphs on complete
graphs, on cycles, and on paths. Fazil et al.
\cite{fazilfixing,fazildist} have studied fixing number and
distinguishing number of functigraphs. The aim of this paper is to
study the variation of location-domination number in passing from a
graph to its functigraph and to find its sharp lower and upper
bounds.

Formally, a functigraph is defined as: Let $G_{1}$ and $G_{2}$ be
the disjoint copies of a connected graph $G$ and let
$f:V(G_{1})\rightarrow V(G_{2})$ be a function. A \emph{functigraph}
$F^f_{G}$ of the graph $G$ consists of the vertex set $V(G_{1})\cup
V(G_{2})$ and the edge set $E(G_{1})\cup E(G_{2})\cup
\{uv:v=f(u)\}$. Unless otherwise specified, all the graphs $G$
considered in this paper are simple, non-trivial and connected.
Throughout the paper, we will denote $V(G_{1})=A_1$, $V(G_{2})=A_2$,
$f(V(G_{1}))= I$, $|I|= k$, a locating-dominating set of $F^f_{G}$
with the minimum cardinality by $L_D^{\ast}$, the elements of $A_1$
and $A_2$ are denoted by $u$ and $v$, respectively and each section
of this paper has different labeling for the elements of $A_1$ and
$A_2$.

This paper is organized as follows. Section 2 gives the sharp lower
and upper bounds for the location-domination number of functigraphs.
This section also establishes the connection between the
location-domination number of graphs and their corresponding
functigraphs in the form of realizable result. Section 3 provides
the location-domination number of functigraphs of the complete
graphs for all possible definitions of the function $f$. In Section
4, we investigate the location-domination number of the functigraph
of a family of spanning subgraphs of the complete graphs for all
possible definitions of constant function $f$.
\section{Some basic results and bounds}

By the definitions of twin vertices and twin-set, we have the
following straightforward result:

\begin{Proposition}\label{Prop2}\cite{Murtaza}
Let $T$ be a twin-set of cardinality $m\geq 2$ in a connected graph
$G$. Then, every locating-dominating set $L_D$ of $G$ contains at
least $m-1$ vertices of $T$.
\end{Proposition}

\begin{Theorem}\cite{slater1}\label{ThmCompleteLDN}
Let $G$ be a graph of order $n\ge 2$, then $\lambda(G)=n-1$ if and
only if $G=K_n$ or $G=K_{1,n-1}$, where $K_n$ and $K_{1,n-1}$ are
the complete graph and complete bipartite graph of order $n$.
\end{Theorem}

\begin{Lemma}\label{lemmaLowerboundlamda}
Let $G$ be a graph of order $n\ge 2$ and $F_G^f$ be its
corresponding functigraph. If $\lambda$ is the location-domination
number of $F_G^f$, then $ 2n+1 \le 2^{\lambda}+\lambda$.
\end{Lemma}

\begin{proof}
Let $L\subset V(F_G^f)$ be a non-empty set and $|L|=\lambda$. Let
$v\in V(F_G^f)\setminus L$, then $N(v)\cap L$ is a subset of $L$. If
$L$ is a locating-dominating set of $F_G^f$, then $N(v)\cap L$ for
all $v\in V(F_G^f)\setminus L$ must be non-empty distinct subsets of
$L$, which is possible only when the number of non-empty subsets of
$L$ are greater than or equal to the number of vertices in
$V(F_G^f)\setminus L$. Since the number of non-empty subsets of $L$
is $2^{\lambda}-1$ and the number of vertices in $V(F_G^f)\setminus
L$ is $2n-\lambda$. Therefore the result follows.
\end{proof}

\begin{figure}
  \includegraphics[width=14cm]{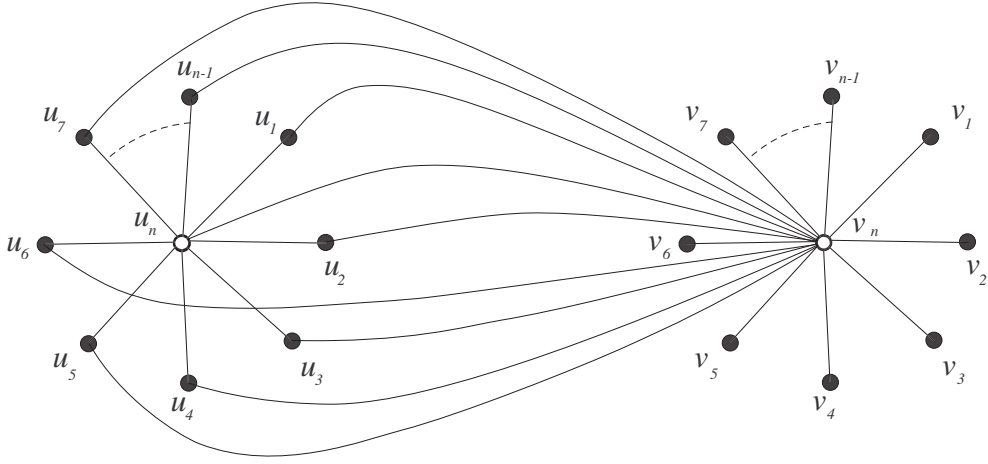}
  \caption{The functigraph of $K_{1,n-1}$ when $f$ is constant and
  $I=\{v_n\}$. Black vertices form a locating-dominating set with the minimum cardinality.}
  \label{Fig1}
\end{figure}

\begin{Theorem}\label{Bounds on lambda}
Let $G$ be a graph of order $n\ge 3$, then $3\le \lambda (F_G^f)\le
2n-2$. Both bounds are sharp.
\end{Theorem}

\begin{proof}
Since $n\ge 3$, therefore by Lemma \ref{lemmaLowerboundlamda}, $7\le
2^{\lambda}+\lambda$ which yields $3\le \lambda$. For the sharpness
of the lower bound, let $G=P_3$ be the path graph of order 3 and $f$
be identity function, then $\lambda(F_G^f)=3$. For the upper bound,
we consider the most worse cases in which $\lambda(G)=n-1$ and $f$
is a constant function. If $\lambda(G)=n-1$, then by Theorem
\ref{ThmCompleteLDN}, $G$ is either $K_n$ or $K_{1,n-1}$. It is
proved in Lemma \ref{LemmaCompleteConstant} that $\lambda
(F_G^f)=2n-3$, whenever $G=K_n$ and $f$ is a constant function.
Therefore, we consider $G=K_{1,n-1}$ and $f$ is constant. Let
$V(G_1)=A_1=\{u_1,...,u_{n-1}\}\cup \{u_n\}$ where each of the
vertices $u_1,...,u_{n-1}$ is adjacent to $u_n$. Similarly, label
the corresponding vertices of $A_2=\{v_1,...,v_{n-1}\}\cup \{v_n\}$.
We define a constant function $f:A_1\rightarrow A_2$ by $f(u_i)=v_n$
for all $1\le i\le n$. The corresponding functigraph $F_G^f$ is
shown in the Figure \ref{Fig1}. Our claim is $\lambda (F_G^f)=2n-2$.
Since $\{u_1,...,u_{n-1}\}\cup \{v_1,...,v_{n-1}\}$ is a
locating-dominating set of $F_G^f$, therefore $\lambda (F_G^f)\le
2n-2$. Let $L_D$ be a locating-dominating set of $F_G^f$. Since
$F_G^f$ contains disjoint twin sets $\{u_1,...,u_{n-1}\}$ and
$\{v_1,...,v_{n-1}\}$, therefore by Proposition \ref{Prop2}, $L_D$
must contains at least $n-2$ vertices from each of these twin sets
and hence $\lambda (F_G^f)\ge 2n-4$. Without loss of generality,
assume $L_D$ contains $\{u_1,...,u_{n-2}\}$ and
$\{v_1,...,v_{n-2}\}$ from each of these twin sets. We claim that
$L_D$ contains at least two vertices from the set
$B=\{u_{n-1},u_n,v_{n-1},v_n\}$. If $|L_D\cap B|=0$, then
$N(u_{n-1})\cap L_D=N(v_{n-1})\cap L_D=\emptyset$, a contradiction.
If $|L_D\cap B|=1$, then there are the following possible cases. If
$L_D\cap B=\{u_{n-1}\}$ or $L_D\cap B=\{u_{n}\}$, then
$N(v_{n-1})\cap L_D=\emptyset$, a contradiction. If $L_D\cap
T=\{v_{n}\}$, then $N(u_{n-1})\cap L_D=N(v_{n-1})\cap L_D$, a
contradiction. If $L_D\cap B=\{v_{n-1}\}$, then $N(u_{n-1})\cap
L_D=\emptyset$, a contradiction. Thus, $|L_D\cap B|\ge 2$ and
consequently $|L_D|\ge 2n-2$. Hence, $\lambda(F_G^f)=2n-2$ and the
result follows.
\end{proof}
\begin{Lemma}
For any integer $t\ge 2$, there exist a connected graph $G$ such
that $\lambda(F_G^f)-\lambda(G)=t$.
\end{Lemma}
\begin{proof}
We construct the graph $G$ by taking the path graph $P_3$ and label
its vertices as $u_1,u_2$ and $u_3$. Attach $t-1$ pendants with
$u_1$ and label them as $u_{1,i}$ where $1\le i \le t-1$. This
completes the construction of the graph $G$. Take another copy of
$G$ and label the corresponding vertices with $v_1,v_2,v_3$ and
$v_{1,i}$ where $1\le i \le t-1$. Define a constant function
$f:A_1\rightarrow A_2$ which maps every vertex of $A_1$ to $v_1\in
A_2$. First we prove that $\lambda(G)=t$. Consider the set
$\{u_1,u_3,u_{1,1},...,u_{1,t-2}\}$, then the reader can easily
verify that this is a locating-dominating set of cardinality $t$ and
hence $\lambda(G)\le t$. Let $L_D$ be a locating-dominating set of
$G$. Since $G$ contains $\{u_{1,1},...,u_{1,t-1}\}$ twin vertices,
therefore by Proposition \ref{Prop2} $\lambda(G)\ge t-2$. Without
loss of generality, assume $L_D\cap
\{u_{1,1},...,u_{1,t-1}\}=\{u_{1,1},...,u_{1,t-2}\}$. Our claim is
$L_D$ contains atleast two elements from
$B=\{u_1,u_2,u_3,u_{1,t-1}\}$. If $|L_D\cap B|=0$, then
$N(u_{1,t-1})\cap L_D=\emptyset$, a contradiction. If $|L_D\cap
B|=1$, then we discuss four possible cases. If $L_D\cap
B=\{u_{1,t-1}\}$, then $L_D\cap N(u_2)=\emptyset$, a contradiction.
If $L_D\cap B=\{u_1\}$, then $L_D\cap N(u_3)=\emptyset$, a
contradiction. If $L_D\cap B=\{u_2\}$ or $L_D\cap B=\{u_3\}$, then
$L_D\cap N(u_{1,t-1})=\emptyset$, a contradition. Thus, $|L_D\cap
B|\ge 2$ and consequently $|L_D|\ge t$. Thus, $\lambda(G)=t$.

Next we prove that $\lambda(F_G^f)=2t$. Consider the set
$\{u_1,u_3,u_{1,1},...,u_{1,t-2},v_1,v_3,$
$v_{1,1},...,v_{1,t-2}\}$, then the reader can easily verify that
this is a locating-dominating set of $F_G^f$ of cardinality $2t$ and
hence $\lambda(F_G^f)\le 2t$. Since $f$ is a constant function,
therefore the sets $\{u_{1,1},...,u_{1,t-1}\}$ and
$\{v_{1,1},...,v_{1,t-1}\}$ are also disjoint twin sets of $F_G^f$
each of cardinality $t-1$, therefore by Proposition \ref{Prop2},
$\lambda(F_G^f)\ge 2t-4$. Let $L_D$ be a locating-dominating set of
$F_G^f$. Without loss of generality, assume $L_D\cap
\{u_{1,1},...,u_{1,t-1}\}=\{u_{1,1},...,u_{1,t-2}\}$ and $L_D\cap
\{v_{1,1},...,v_{1,t-1}\}=\{v_{1,1},...,v_{1,t-2}\}$. As $f$ is a
constant function, therefore by using the similar arguments as in
the case of graph $G$, the locating-dominating set $L_D$ of $F_G^f$
must contains atleast two elements from each of the sets
$\{u_1,u_2,u_3,u_{1,t-1}\}$ and $\{v_1,v_2,v_3,v_{1,t-1}\}$ and
consequently, $|L_D|\ge 2t$. Thus, $\lambda(F_G^f)=2t$ and the
result follows.
\end{proof}
\section{The location-domination number of functigraphs of the complete graphs}
We find the location-domination number of functigraph of the
complete graphs for all possible definitions of the function $f$. In
this section, we use the following terminology for labeling the
vertices of functigraph. Let $G$ be a complete graph of order $n$
and $f:A_1\rightarrow A_2$ be a function. Let $v\in I\subset A_2$,
then we denote the set $\{f^{-1}(v)\}\subset A_1$ by $\Psi_{v}$ and
its cardinality by $s=|\Psi_{v}|$ $(1\le s\le n)$. If $s=1$ for some
$v\in I$, then we name the edge $vf^{-1}(v)\in E(F_G^f)$ as a
\emph{functi matching} of $F_G^f$. The discussion has two parts, the
first part discuss the cases in which $F_G^f$ does not have any
functi matching and in the second part $F_G^f$ have at least one
functi matching.
\begin{figure}
  \includegraphics[width=14cm]{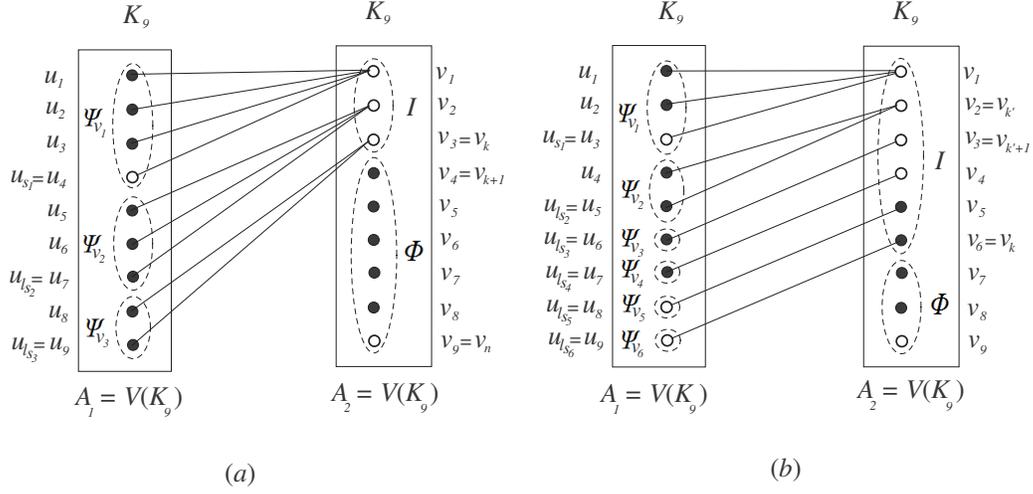}
  \caption{The labeling of the vertices of functigraph of $K_9$ when ($a$) $F_G^f$
  does not have any functi matching for $k=3$  ($b$) $F_G^f$
  has functi matching for $k=6$ and $k'=2$. Black vertices form a locating-dominating
  set of $F^f_G$ with the minimum cardinality.}
  \label{Fig2}
\end{figure}
 For the first part of discussion, let $F_G^f$ does not have any functi
 matching. In this case, we label
  the vertice of $I$
as: $I=\{v_1, v_2,...,v_{k}\}$ where the subscript index is assigned
to each $v$ according to the index of corresponding
$s_i=|\Psi_{v_i}|$ $(1\le i \le k)$, where $s_i$ are assinged
indices according as $s_1\ge s_2 \ge ...\ge s_k$. The set
$A_2\setminus I$ is a twin set of $F_G^f$ and we denote the set by
$\Phi=A_2\setminus I$. The vertices of $\Phi$ are labeled as
$\Phi=\{v_{k+1},v_{k+2},...,v_n\}$. The vertices of $A_1$ are
labeled as: $\Psi_{v_1}=\{u_1,u_2...,u_{s_1}\}$ and for each
$i=2,...,k$, $\Psi_{v_i}=\{u_{l_1},u_{l_2},...,u_{l_{s_{i}}}\}$,
where the indices $l_j$ $(j=1,2,...,s_i)$ for fix $i$ are given by
$l_j=\sum_{m=1}^{i-1}s_m+j$. The labeling of vertices of a
functigraph of the complete graph $K_9$ for $k=3$ is illustrated in
the Figure \ref{Fig2}$(a)$ where the functigraph does not have any
functi matching. It can be seen that for each $i$ $(1\le i\le k)$,
$s_i\ge 2$, $\Psi_{v_i}\subset A_1$ is a twin-set of vertices. Also,
$\cup_{i=1}^k \Psi_{v_i}=A_1$ and $\sum_{i=1}^k s_i=n$.
\begin{Lemma}\label{LemmaCompleteConstant} Let $G=K_{n}$ be the complete
 graph of order $n\geq 2$, and $f:A_1\rightarrow A_2$ be a constant function, then
$$\lambda(F_G^f)=\left\{
\begin{array}{ll}
2n-2,\,\,\,\,\,\,\,\,\,\,\,\,\,\,\,\,\,\,\,if \,\,\,\, n=2 &  \\
2n-3,\,\,\,\,\,\,\,\,\,\,\,\,\,\,\,\,\,\,\,if \,\,\,\, n\ge 3 \\
\end{array}
\right.$$
\end{Lemma}
\begin{proof}
For $n=2$, $F_G^f$ is the complete graph $K_3$ with a pendant
attached with any one of the vertices of $K_3$. Clearly,
$\lambda(F_G^f)=2n-2$. For $n\ge 3$, $F_G^f$ with constant $f$ has
$I=\{v_1\}$ and using the labeling as defined earlier
$\Psi_{v_1}=\{u_1,...,u_n\}$ and $\Phi=\{v_2,...,v_n\}$. Let $L_D^*$
be a locating-dominating set of $F_G^f$ with the minimum
cardinality. By Proposition \ref{Prop2}, $|L_D^*\cap \Psi_{v_1}|\ge
n-1$ and $|L_D^*\cap \Phi|\ge n-2$. Thus, $\lambda(F_G^f)\ge 2n-3$.
Moreover, $A_1\setminus\{u_n\}\cup A_2\setminus\{v_1,v_n\}$ is
locating-dominating set of $F_G^f$, and hence $\lambda(F_G^f) \le
2n-3$ and the result follows.
\end{proof}
\begin{Theorem}\label{f12}
 Let $G$ be the complete graph of order $n\geq
4$ and $F_G^f$ does not has functi matchings. If $1<k<n$, then
$\lambda(F_G^f)=2n-k-2$
\end{Theorem}
\begin{proof}
For $1<k<n$, $I=\{v_1, v_2,...,v_{k}\}$. First we prove that the set
$L=\{\cup_{i=1}^{k} \Psi_{v_i}\setminus \{u_{s_1}\}\}\cup
\{\Phi\setminus\{v_n\}\}$ is a locating-dominating set of $F_G^f$
with the cardinality $|L|=\sum_{i=1}^{k}s_i-1+(n-k-1)=2n-k-2$. Since
$V(F_G^f)\setminus L=\{u_{s_1},v_1,...,v_{k},v_n\}$. We prove that
all the elements of $V(F_G^f)\setminus L$ have distinct non-empty
neighbors in $L$. Now, $N(u_{s_1})\cap L=\cup_{i=1}^{k}
\Psi_{v_i}\setminus \{u_{s_1}\}$, $N(v_1)\cap
L=\{\Psi_{v_1}\setminus \{u_{s_1}\}\}\cup\{\Phi\setminus\{v_n\}\}$,
for each $i$ where $2\le i \le k$, $N(v_i)\cap
L=\Psi_{v_i}\cup\{\Phi\setminus{v_n}\}$, $N(v_n)\cap L=
\Phi\setminus \{v_n\}$. Thus, $L$ is a locating-dominating set of
$F_G^f$. Hence, $\lambda(F_G^f)\le 2n-k-2$. Let $L_D^*$ be a
locating-dominating set of $F_G$ with the minimum cardinality. Then
 by Proposition \ref{Prop2}, $L_D^*$ must contains $s_i-1$ vertices
 of the disjoint twin sets
$\Psi_{v_i}$ for each $i$, $1< i \le k$ and $n-k-1$ vertices of the
twin set $\Phi$. Therefore, $\lambda(F_G^f)\ge 2n-2k-1$. Without
loss of generality, assume the set $L_D^*\cap
\Psi_{v_1}=\Psi_{v_1}\setminus \{u_{s_1}\}$, $L_D^*\cap
\Psi_{v_i}=\Psi_{v_i}\setminus \{u_{l_{s_i}}\}$ for each $i$ $(2\le
i \le k)$ and $L_D^*\cap \Phi=\Phi\setminus \{v_n\}$. Now consider
the two element sets $\{u_{s_1},v_1\}$ and $\{u_{l_{s_i}},v_i\}$
$(2\le i \le k)$. We claim that $L_D^*$ contains atleast one element
from exactly $k-1$ sets of these two element sets. Consider
$u_{s_1},v_1\not\in L_D^*$. Next we prove that one vertex from
$\{u_{l_{s_i}},v_i\}$ $(2\le i \le k)$ must belongs to $L_D^*$ for
all $i$ $(2\le i \le k)$. Suppose on contrary that both $v_i$ and
$u_{l_{s_i}}$ do not belong to $L_D^*$ for some $i$ ($2\le i \le
k$). Then $N(u_{l_{s_i}})\cap L_D^*= N(u_{s_1})\cap L_D^*$, a
contradiction. Similarly, by considering $u_{l_{s_i}},v_i\not\in
L_D^*$ for some $i$ $(2\le i \le k)$ and using similar arguments we
leads to a contradiction. Thus, $L_D^*$ must contains atleast one
element from exactly $k-1$ sets of these two element sets.
Consequently, $|L_D^*|\ge 2n-k-2$. Hence, $\lambda(F_G^f)=2n-k-2$.
\end{proof}
For the second part of discussion, let $F_G^f$ has atleast one
functi matching. In this case, we label the vertices of $I$ as:
$I=\{v_1, v_2,...,v_{k'},v_{k'+1},...,v_{k}\}$ where $(1\le k' < k)$
and the subscript index is assigned to each $v$ according to the
index of corresponding $s_i$ $(1\le i \le k)$, where $s_i$ are
assinged indices according as $s_1\ge s_2 \ge ...\ge s_{k'}>
s_{k'+1}=...=s_k=1$. Notations of $\Phi$ and $\Psi_{v_i}$ $(1\le i
\le k')$ are same as used earlier. The labeling of vertices of a
functigraph of the complete graph $K_9$ for $k=6$ and $k'=2$ is
illustrated in the Figure \ref{Fig2}$(b)$ where the functigraph has
four functi matchings. It can be seen that for each $i$ $(k'+1\le
i\le k)$, $s_i=1$ and $\Psi_{v_i}=\{u_{{l_{s_{k'}}+i}}\}$ and the
edge $u_{{l_{s_{k'}}+i}}v_i\in E(F_G^f)$ is a functi matching of
$F_G^f$.
\begin{Lemma}\label{LemmaCompleteBijective} Let $G=K_{n}$ be the complete
 graph of order $n\geq 2$ and $f:A_1\rightarrow A_2$ be a bijective function, then
$$\lambda(F_G^f)=\left\{
\begin{array}{ll}
n,\,\,\,\,\,\,\,\,\,\,\,\,\,\,\,\,\,\,\,if \,\,\,\, n=2,3 &  \\
n-1,\,\,\,\,\,\,\,\,\,if \,\,\,\, n\ge 4 \\
\end{array}
\right.$$
\end{Lemma}
\begin{proof}
For $n=2$, $F_G^f$ is a cyclic graph of order 4 and hence
$\lambda(F_G^f)=2$ by \cite{cac2}. For $n=3$, $F_G^f$ is a
triangular prism of order 6 and $\lambda(F_G^f)\ge 3$ by Lemma
\ref{lemmaLowerboundlamda}, whereas $A_1$ is a locating-dominating
set of $F_G^f$ with the cardinality 3. For $n\ge 4$, we have the
labeling as defined earlier $I=A_2=\{v_1,...,v_n\}$ and
$\Psi_{v_i}=\{u_i\}$ for all $i$ $(1\le i \le n)$ and hence
$A_1=\{u_1,...,u_n\}$. Consider the set $L=\{u_1,...,u_{n-2},v_n\}$.
We prove that $L$ forms a locating-dominating set of $F_G^f$. As
$N(u_{n-1})\cap L=\{u_1,...,u_{n-2}\}$, $N(u_{n})\cap
L=\{u_1,...,u_{n-2},v_n\}$. Also, $N(v_i)\cap L=\{u_i,v_n\}$ for all
$1\le i \le n-2$ and $N(v_{n-1})\cap L=\{v_n\}$. Thus, $L$ forms a
locating-dominating set of $F_G^f$ and $\lambda(F_G^f)\le n-1$. Next
we prove that $\lambda(F_G^f)\ge n-1$. Suppose on contrary there
exist a locating-dominating set $L_D$ of cardinality $n-2$, then
either $|L_D\cap A_1|\le n-2$ or $|L_D\cap A_2|\le n-2$. Assume
$|L_D\cap A_1|= n-2-j$ and $|L_D\cap A_2|=j$ where $0\le j\le n-2$.
Since $u_iv_i$ forms functi matching of $F_G^f$ for all $i$ $(1\le i
\le n)$. Therefore without loss of generality assume that $L_D\cap
A_1=\{u_1,u_2,...,u_{n-2-j}\}$. If $L_D$ is a locating-dominating
set of $F_G^f$, then the vertices of $A_1\setminus
L_D=\{u_{n-1-j},u_{n-j},...,u_{n}\}$ must have distinct neighbors in
$L_D$ which is possible only $|L_D\cap A_2|= j+2$, a contradiction.
Thus, $\lambda(F_G^f)\ge n-1$ and the result follows.
\end{proof}

%
%
%
%
%
%
%

\begin{Theorem}\label{f13}
Let $G$ be the complete graph of order $n\geq 3$ and $F_G^f$ has
atleast one functi matching. If $1<k<n$, then

$$\lambda(F_G^f)=\left\{
\begin{array}{ll}
2n-k-1,\,\,\,\,\,\,\,\,\,\,\,\,\,\,\,\,\,\,\,if \,\,\,\, n=3, \,\,\,\, k=2, &  \\
2n-k-2,\,\,\,\,\,\,\,\,\,\,\,\,\,\,\,\,\,\,\,if \,\,\,\, n\ge 4, \,\,\,\, k\ge 2.&  \\
\end{array}
\right.$$
\end{Theorem}
\begin{proof}(i) For $n=3$ and $k=2$, let $f:A_1\rightarrow A_2$ be defined as
$f(u_i)=v_1$, where $i=1,2$ and $f(u_3)=v_2$. By Lemma
\ref{lemmaLowerboundlamda}, $\lambda(F_G^f)> 2$. Moreover,
$\{u_1,u_3,v_2\}$ forms a locating-dominating set of $F_G^f$. Thus,
$\lambda(F_G^f)= 3$.

\noindent(ii) For $n\ge 4$ and $2 \le k \le n-1$. Since $F_G^f$ has
functi matchings, therefore there exists a $k'$ $(1\le k' < k)$ such
that $s_{i}=1$ for all $i$ $(k'+1\le i\le k)$. Thus we use the
labeling as described earlier for the vertices of $A_1$ and $A_2$.
Let $L_D^*$ be a locating-dominating set of $F_G^*$ with the minimum
cardinality. The proof consists of the following claims:
\begin{enumerate}
\item\textbf{Claim.} The set $L=\{\cup_{i=1}^{k'}
\Psi_{v_i}\setminus
\{u_{s_1}\}\}\cup\{u_{l_{s_{k'}}+1},u_{l_{s_{k'}}+2},...,u_{l_{s_{k'}}+k-1},v_k\}\cup
\{\Phi\setminus\{v_n\}\}$ is a locating-dominating set of the
cardinality
$|L|=\sum_{i=1}^{k'}s_i-1+(n-\sum_{i=1}^{k'}s_i)+(n-k-1)=2n-k-2$.

\noindent\emph{\textbf{Proof of claim}}: Since $V(F_G^f)\setminus
L=\{u_{s_1},u_{l_{s_{k'}}+k},v_1,...,v_{k'},v_{k'+1},...,v_{k-1},v_n\}$.
We prove that all the elements of $V(F_G^f)\setminus L$ have
distinct neighbors in $L$. Now, $N(u_{s_1})\cap L=A_1\setminus
\{u_{s_1},u_{l_{s_{k'}}+k}\}$, $N(u_{l_{s_{k'}}+k})\cap
L=A_1\setminus \{u_{s_1},u_{l_{s_{k'}}+k}\}\cup \{v_k\}$,
$N(v_1)\cap L=\{\Psi_{v_1}\setminus \{u_{s_1}\}\}\cup
\{v_k\}\cup\{\Phi\setminus\{v_n\}\}$, for each $i$ where $2\le i \le
k'$, $N(v_i)\cap L=\Psi_{v_i}\cup
\{v_k\}\cup\{\Phi\setminus\{v_n\}\}$, for each $i$ where $k'+1\le i
\le k-1$, $N(v_i)\cap L=\{u_{l_{s_{k'}}+i-k'}\}\cup
\{v_k\}\cup\{\Phi\setminus \{v_n\}\}$, $N(v_n)\cap L=
\{v_k\}\cup\{\Phi\setminus \{v_n\}\}$. Thus, $L$ is a
locating-dominating set of $F_G^f$ and hence, $\lambda (F_G^f)\le
2n-k-2$.

\item \textbf{Claim.} $\lambda(F_G^f)\ge n+\sum_{i=1}^{k'}s_i-k-2$.

\noindent\emph{\textbf{Proof of claim}}: By Proposition \ref{Prop2},
$L_D^*$ must contains $s_i-1$ vertices
 of the disjoint twin sets
$\Psi_{v_i}$ for each $i$, $1< i \le k'$ and $n-k-1$ vertices of the
twin set $\Phi$. Therefore, $\lambda(F_G^f)\ge
\sum_{i=1}^{k'}(s_i-1)+(n-k-1)=n+\sum_{i=1}^{k'}s_i-k'-k-1$. Without
loss of generality, assume the set $L_D^*\cap
\Psi_{v_1}=\Psi_{v_1}\setminus \{u_{s_1}\}$, $L_D^*\cap
\Psi_{v_i}=\Psi_{v_i}\setminus \{u_{l_{s_i}}\}$ for each $i$ $(2\le
i \le k')$ and $L_D^*\cap \Phi=\Phi\setminus \{v_n\}$. Now consider
the two element sets $\{u_{s_1},v_1\}$ and $\{u_{l_{s_i}},v_i\}$
$(2\le i \le k')$. We claim that $L_D^*$ must contains atleast one
element from exactly $k'-1$ sets of these two element sets. Consider
$\{u_{s_1},v_1\}\not\subset L_D^*$. We prove that one vertex from
$\{u_{l_{s_i}},v_i\}$ $(2\le i \le k')$ must belong to $L_D^*$ for
all $i$ $(2\le i \le k')$. Suppose on contrary that both
$u_{l_{s_i}}$ and $v_i$
 do not belong to $L_D^*$ for some $i$ ($2\le i \le
k'$). Then $N(u_{l_{s_i}})\cap L_D^*= N(u_{s_1})\cap L_D^*$, a
contradiction. Similarly, by considering
$\{u_{l_{s_i}},v_i\}\not\subset L_D^*$ for some $i$ $(2\le i \le
k')$ and using the similar arguments we lead to a contradiction.
Thus $L_D^*$ must contains atleast one element from exactly $k'-1$
sets of these two element sets. Consequently, $\lambda(F_G^f)\ge
n+\sum_{i=1}^{k'}s_i-k-2$.

\item \textbf{Claim.} $|L_D^*\cap \{u_{l_{s_{k'}}+1},u_{l_{s_{k'}}+2},...,
u_{k},v_{k'+1}, v_{k'+2},...,v_{k}\}|= n-\sum_{i=1}^{k'}s_i$.

\noindent\emph{\textbf{Proof of claim}}: As
$u_{l_{s_{k'}}+i}v_{k'+i}\in E(F^f_G)$ for each $i$ $(1\le i \le
k-k')$ form the functi matchings of $F^f_G$. We take the assumptions
that we have proved in Claim 2 that $L_D^*\cap
\Psi_{v_1}=\Psi_{v_1}\setminus \{u_{s_1}\}$, $L_D^*\cap
\Psi_{v_i}=\Psi_{v_i}$ for each $i$ $(2\le i \le k')$ and $L_D^*\cap
\Phi=\Phi\setminus \{v_n\}$. Now consider the two element sets
$\{u_{l_{s_{k'}}+i},v_{i}\}$ $(k'+1\le i \le k)$ as the sets of two
end vertices of the functi matchings of $F_G^f$. We prove that
$L_D^*$ must contains exactly one element from each of these $k-k'$
sets. Suppose on contrary $\{u_{l_{s_{k'}}+i},v_{i}\}\not\subset
L_D^*$ for some $i$ $(k'+1\le i \le k)$, then
$N(u_{l_{s_{k'}}+i})\cap L_D^*= N(u_{s_1})\cap L_D^*$, a
contradiction. Thus either $u_{l_{s_{k'}}+i}$ or $v_{i}$ for each
$i$ $(k'+1\le i \le k)$ (not both to maintain the minimality of
$L_D^*$) must belong to $L_D^*$ to make distinct neighbor of
$u_{l_{s_{k'}}+i}$ in $L_D^*$. Hence, $|L_D^*\cap
\{u_{l_{s_{k'}}+1},u_{l_{s_{k'}}+2},..., u_{k},v_{k'+1},
v_{k'+2},...,v_{k}\}|=k-k'= n-\sum_{i=1}^{k'}s_i$.

\end{enumerate}
Since the sets used to prove Claim 2 and Claim 3 are disjoint
subsets of $V(F_G^f)$, therefore combining Claim 2 and Claim 3 we
get $\lambda(F_G^f)\ge 2n-k-2$. Combining this with Claim 1 we get
the required result.
\end{proof}

\begin{Corollary}
Let $G$ be the complete graph of order $n\geq 4$ and let $F_G^f$
contains $p$ functi matchings, then $\lambda(F_G^f)\ge p$. The bound
is sharp.
\end{Corollary}

\begin{proof}
The result follows from proof of Claim 3 of Theorem \ref{f13}.
Sharpness of the bound follows from Lemma
\ref{LemmaCompleteBijective} where $f$ is a bijective function.
\end{proof}

\begin{Corollary}
Let $G$ be the complete graph of order $n\geq 4$. Then
$\lambda(G)=\lambda(F_G^f)$ if and only if $k=n-1$.
\end{Corollary}

\begin{proof}
The result follows by Theorem \ref{f13} for $n\ge 4$.
\end{proof}
\section{Location-Domination Number of Functigraph of a family of
spanning subgraphs of the Complete Graphs} A vertex $u\in V(G)$ is
called a \emph{saturated} vertex, if $deg(u)=|V(G)|-1$. Since any
two saturated vertices are adjacent twins, therefore the set of all
saturated vertices of a graph forms a twin set represented by $T^s$.
Let $e'\in E(G)$ be an edge that joins two saturated vertices of
$G$, then the spanning subgraph of $G$ which is obtained by removing
the edge $e'$ is denoted by $G - e'$. Similarly, $H_i=G-ie'$ $(1
\leq i \leq \lfloor \frac{n}{2} \rfloor)$ denotes a spanning
subgraph of $G$ that is obtained by removing $i$ edges $e'$, where
$e'$ joins two saturated vertices of $G$. It may also be noted that
after removing the edge $e'$, the two saturated vertices that are
connected by $e'$, are converted to non-adjacent twins and hence
forms a twin set of cardinality 2. The twin set obtained after
removing the $i$th edge $e'$ is denoted by $T^t_i$, $(1 \leq i \leq
\lfloor \frac{n}{2} \rfloor)$. Thus, $H_i$ $(1 \leq i \leq \lfloor
\frac{n}{2} \rfloor)$ has $i$ twin sets $T^t_i$ of non-adjacent
twins, each of cardinality 2 and one twin set $T^s$ of adjacent
twins of the remaining $n-2i$ saturated vertices. Further, if $n$ is
even and $i=\frac{n}{2}$, then $T^s=\emptyset$. We label the
vertices in $H_i$ as follows: $T^t_i=\{u_i^1,u_i^2\}$ $(1 \leq i
\leq \lfloor \frac{n}{2} \rfloor)$ and
$T^s=\{u_{2i+1},u_{2i+2},...,u_n\}$ $(1\le i < \frac{n}{2})$. The
following theorem gives location-domination number of functigraph of
$G-ie'$.

\begin{Lemma}
Let $G$ be the complete graph of order $n=4$ and $f$ be a constant
function, then $\lambda(F_{H_i}^f)=4$ $(1\le i \le 2)$.
\end{Lemma}

\begin{proof}
Since $f$ is a constant function, therefore assume that
$I=\{v\}\subset A_2$.  If $v\in T^t_j\subset A_2$ for some $j$
$(1\le j \le 2)$, then $T^t_j\subset A_2$ is not twin set in
$F_{H_i}^f$. Similarly if $v\in T^s\subset A_2$, then $T^s\subset
A_2$ is not a twin set in $F_{H_i}^f$. Let a set $L\subset
V(F_{H_i}^f)$, such that $L$ has exactly one element from each twin
set of $F_{H_i}^f$. We discuss the following cases.
\begin{enumerate}
\item If $i=1$ and $v\in T^t_1\subset A_2$, then without loss of
    generality assume $v=v_1^1$. Also, the sets $T^t_1,T^s\subset A_1$
    and $T^s\subset A_2$ are twin sets in the corresponding $F_{H_1}^f$ each
    with the cardinality 2.
    Thus by Proposition \ref{Prop2}, $\lambda(F_{H_1}^f)\ge 3$. There are 8 possible
    choices for the set $L$. If we take $L=\{u_1^1,u_3,v_3\}$, then $N(v_1^2)\cap L=N(v_4)\cap L$ which implies that
     $L$ is not a locating-dominating set. Similarly, the
    other 7 choices for $L$ do not form locating-dominating set and hence, $\lambda(F_{H_1}^f)\ge 4$.
    Also, the set $\{u_1^1,u_3,v_3,v_4\}$ forms a locating-dominating set of
    $F_{H_1}^f$. Hence, the result follows.
\item If $i=1$ and $v\in T^s\subset A_2$,
    then without loss of generality assume $v=v_3$. The corresponding
    functigraph $F_{H_1}^f$ has twin sets $T^s,T_1^t\subset A_1$
    and $T^t_1\subset A_2$ each of the cardinality 2. Then it can be seen that
    8 possible choices of the set $L$ do not form locating-dominating set
     of $F_{H_1}^f$, therefore
    $\lambda(F_{H_1}^f)\ge 4$. Also, the set $\{u_1^1,u_3,v_1^1,v_4\}$
    forms a locating-dominating set of $F_{H_1}^f$.
\item If $i=2$ and $v\in T_1^t\subset A_2$, then $T^s=\emptyset$.
    Without loss of generality assume $v=v_1^1$. The corresponding
    functigraph $F_{H_2}^f$ has twin sets $T_1^t,T_2^t\subset A_1$
    and $T^s\subset A_2$ each of the cardinality 2. By Proposition
    \ref{Prop2}, $\lambda(F_{H_1}^f)\ge 3$. There are 8 possible choices
    for the set $L$. If $L=\{u_1^1,u_2^1,v_2^1\}$, then
    $N(v_2^2)\cap L=\emptyset$. If $L=\{u_1^1,u_2^1,v_2^2\}$,
    then $N(v_2^1)\cap L=\emptyset$. Thus, $L$ is not locating-dominating
     set. Similarly, the remaining 6 choices for the set $L$ do not form
     locating-dominating set. Hence, $\lambda(F_{H_2}^f)\ge 4$.
    Also the set $\{u_1^1,u_2^1,v_2^1,v_2^2\}$ forms a
    locating-dominating set of $F_{H_2}^f$. The case when $i=2$ and
    $v\in T_2^t\subset A_2$ can also be proved by the similar arguments.
\end{enumerate}
\end{proof}

\begin{figure}
  \includegraphics[width=14cm]{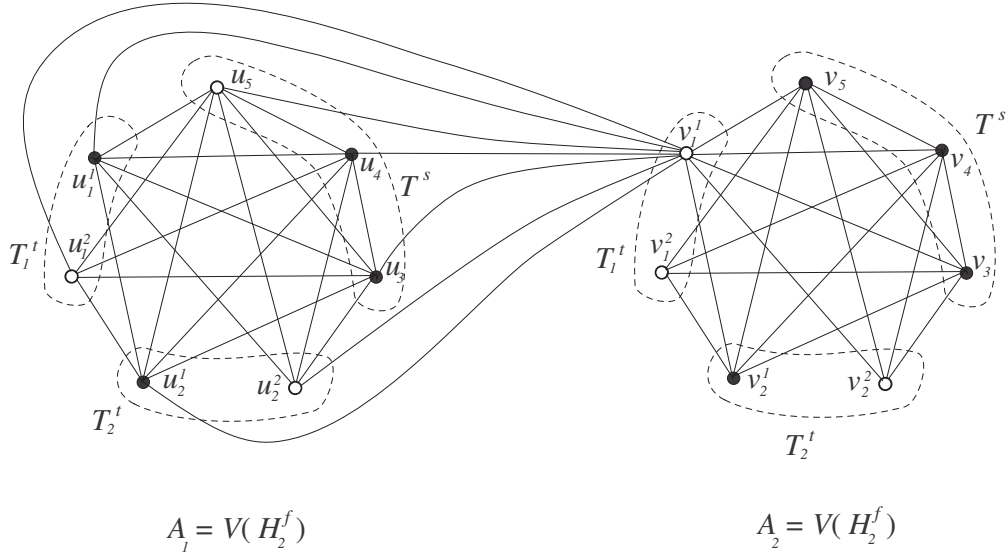}
  \caption{The labeling of the vertices of $F_{H_i}^f$ when $n=7$, $i=2$,  $f$
  is a constant function, $I=\{v\}$ and $v$ is a non-saturated vertex of $A_2$.
  Black vertices form a locating-dominating
  set of $F^f_{H_i}$ with the minimum cardinality.}
  \label{Fig3}
\end{figure}

\begin{Theorem}
Let $G$ be the complete graph of order $n\ge 5$ and $f$ be a
constant function such that $I=\{v\}\subset A_2$, then
$$\lambda(F_{H_i}^f)=\left\{
\begin{array}{ll}
2n-2i-3, \hspace{0.2cm} if \hspace{0.3cm} 1
\leq i \leq \lfloor \frac{n}{2} \rfloor-1 \text{ and v is a saturated vertex of $A_2$} &\\
2n-2i-2, \hspace{0.2cm} if \hspace{0.3cm} 1
\leq i \leq \lfloor \frac{n}{2} \rfloor-1 \text{ and v is a non-saturated vertex of $A_2$}&\\
n-1 \hspace{1.5cm} if \hspace{0.3cm} n \hspace{.1cm} \text{is even
and}
 \hspace{0.1cm}i= \frac{n}{2}&  \\
2\lfloor \frac{n}{2} \rfloor \hspace{1.6cm} if \hspace{0.3cm} n
\hspace{.1cm} \text{is odd and} \hspace{.1cm} i=\lfloor \frac{n}{2}
\rfloor.&
\end{array}
\right.$$
\end{Theorem}

\begin{proof}
Since $f$ is a constant function, therefore the collection
$\{T^t_1,T^t_2,...,T^t_i, T^s\}$ of twin subset of $A_2$ are also
twin sets in the corresponding $F_{H_i}^f$. If $v\in T^t_j\subset
A_2$ for some $j$ $(1\le j \le i)$, then $T^t_j\subset A_2$ is not
twin set in $F_{H_i}^f$. Similarly if $v\in T^s\subset A_2$, then
$T^s\subset A_2$ is not twin set in $F_{H_i}^f$. We discuss the
following cases
\begin{enumerate}
    \item When $1\leq i \leq \lfloor \frac{n}{2}
    \rfloor-1$ and $v\in T^t_j\subset A_2$ for some
    $j$ $(1\le j \le i)$. Then without loss of generality assume that
    $v=v_j^1$. The corresponding $F_{H_i}^f$ have twin sets
    $T^t_1,T^t_2,...,T^t_i, T^s\subset A_1$ and
    $T^t_1,...,T^t_{j-1},T^t_{j+1},...,T^t_i, T^s\subset A_2$. By
    Proposition \ref{Prop2}, $\lambda(F_{H_i}^f)\ge 2i-1+2(n-2i-1)=2n-2i-3$.
    Let a set $L\subset V(F_{H_i}^f)$, such that $L$ has all the elements except
    one element of each of these twin subsets. There are $2^{2i-1}(n-2i-1)^2$
     choices for choosing the elements of the set $L$. Each choice for the set $L$ does not contain an element of the set $T^s\subset
     A_2$. Without loss of generality assume that $v_n\not\in L$, then $N(v_j^2)\cap L=N(v_n)\cap L$.
     Thus, $L$ is not a locating-dominating set of $F_{H_i}^f$ for all choices of the set $L$.
     Hence, $\lambda(F_{H_i}^f)\ge 2n-2i-2$.
     Also the set $\{u_1^1,u_2^1,...,u_i^1,u_{2i+1},...,u_{n-1},v_1^1,v_2^1,
    ...,v_{j-1}^1,v_{j+1}^1,...,v_i^1,v_{2i+1},$ $...,v_{n}\}$ is a
    locating-dominating set with cardinality $2n-2i-2$. Thus
    $\lambda(F_{H_i}^f)=2n-2i-2$.
\item When $1\leq i \leq \lfloor \frac{n}{2}
    \rfloor-1$ and $v\in T^s\subset A_2$. Then without loss of generality assume that
    $v=v_n$. The corresponding $F_{H_i}^f$ have twin sets
    $T^t_1,T^t_2,...,T^t_i, T^s\subset A_1$ and
    $T^t_1,...,T^t_i, T^s\setminus\{v_n\}\subset A_2$. By
    Proposition \ref{Prop2}, $\lambda(F_{H_i}^f)\ge 2i+(n-2i-1)+(n-2i-2)=2n-2i-3$.
    Also the set $L=\{u_1^1,u_2^1,...,u_i^1,u_{2i+1},...,u_{n-1},v_1^1,$ $v_2^1,
    ...,v_i^1,v_{2i+1},...,v_{n-2}\}$ is a locating-dominating set of $F_{H_i}^f$
     with the cardinality $2n-2i-3$. Thus
    $\lambda(F_{H_i}^f)=2n-2i-3$.
    \item When $i= \frac{n}{2}$ and $n$ is even, then $T^s=\emptyset$ and
    $v\in T^t_j\subset A_2$ for some $j$ $(1\le j \le i)$. Then without loss of generality assume that
    $v=v_j^1$. The corresponding $F_{H_i}^f$ have twin sets $T^t_1,T^t_2,...,T^t_i\subset A_1$ and
    $T^t_1,...,T^t_{j-1},T^t_{j+1},...,T^t_i\subset A_2$. By Proposition \ref{Prop2}, $\lambda(F_{H_i}^f)\ge i+(i-1)=n-1$.
    Also, the set
    $L=\{u_1^1,u_2^1,...,u_i^1,v_1^1,v_2^1,...,v_{j-1}^1$, $v_{j+1}^1,...,v_i^1\}$ is a locating-dominating set
    of $F_{H_i}^f$ with cardinality $n-1$. Thus
    $\lambda(F_{H_i}^f)=n-1$.
\item When $i= \lfloor \frac{n}{2}\rfloor$, $n$ is odd and $v\in T^t_j\subset A_2$ for some
    $j$ $(1\le j \le i)$. In this case,
    $T^s=\{v_n\}$ is not a twin set. Without loss of generality assume that
    $v=v_j^1$. The corresponding $F_{H_i}^f$ have twin sets
    $T^t_1,T^t_2,...,T^t_i\subset A_1$ and
    $T^t_1,...,T^t_{j-1},T^t_{j+1},...,T^t_i\subset A_2$. By Proposition \ref{Prop2},
     $\lambda(F_{H_i}^f)\ge i+(i-1)=2i-1$.
    Let a set $L\subset V(F_{H_i}^f)$, such that $L$ has all the elements except
    one element of each of these twin subsets. Each choice for the set $L$ does not contain
    $v_n$. Then, $N(v_j^2)\cap L=N(v_n)\cap L$ for all choices
    of the set $L$. Thus, $L$ is not a locating-dominating set of $F_{H_i}^f$ for all choices
     of the set $L$ and hence, $\lambda(F_{H_i}^f)\ge 2i$. Also, the set $\{u_1^1,u_2^1$ $,...,u_i^1,v_1^1,v_2^1,
    ...,$ $v_{j-1}^1,v_j,$ $v_{j+1}^1,...,v_i^1\}$ is a
    locating-dominating set with cardinality $2i$. Thus
    $\lambda(F_{H_i}^f)=2\lfloor \frac{n}{2}\rfloor$.

\item When $i= \lfloor \frac{n}{2}\rfloor$, $n$ is odd and $v\in T^s\subset A_2$. In this case, $T^s=\{v_n\}$ is not a twin set and $v=v_n$. The corresponding $F_{H_i}^f$ have twin sets
    $T^t_1,T^t_2,...,T^t_i\subset A_1$ and
    $T^t_1,...,T^t_i\subset A_2$. By
    Proposition \ref{Prop2}, $\lambda(F_{H_i}^f)\ge 2i$. Also, the set $L=\{u_1^1,u_2^1$
    $,...,u_i^1,v_1^1,v_2^1,$ $...,v_i^1\}$ is a
    locating-dominating set with cardinality $2i$. Thus
    $\lambda(F_{H_i}^f)=2\lfloor \frac{n}{2}\rfloor$.

\end{enumerate}

\end{proof}


\begin{thebibliography}{999}

\bibitem{berger}
T. Y. Berger-Wolf, W. E. Hart and J. Saia, Discrete sensor placement
problems in distribution networks, \emph{J. Math. Comp. Modeling},
\textbf{42(13)}(2005), 1385-1396.

\bibitem{ber} N. Bertrand, I. Charon, O. Hudry and A. Lobstein, Identifying and
locating-dominating codes on chains and cycles, \emph{European J.
Combin.}, \textbf{25}(2004), 969-987.

\bibitem{cac2}
J. C$\acute{\text{a}}$ceres, C. Hernando, M. Mora, I. M. Pelayo and
M. L. Puertas, On locating and dominating sets in graphs,
\emph{Workshop de Matemática Discreta Algarve/Andalucía--VI
Encuentro Andaluz de Matemática Discreta}, (2009), 19–-22.

\bibitem{cac1} J. C$\acute{\text{a}}$ceres, C. Hernando, M. Mora, I. M. Pelayo and M. L. Puertas,
Locating dominating codes, \emph{Appl. Math. Comput.},
\textbf{220}(2013), 38–-45.

\bibitem {char2}
I. Charon, O. Hudry and A. Lobstein, Identifying and
locating-dominating codes: NP-completeness results for directed
graphs, \emph{IEEE Trans. Inform. Theory}, \textbf{48}(2002),
2192-2200.

\bibitem{char1}
I. Charon, O. Hudry and A. Lobstein, Minimizing the size of an
identifying or locating-dominating code in a graph is NP-hard,
\emph{Theor. Comput. Sci.}, \textbf{290}(2003), 2109-2120.

\bibitem{char} G. Chartrand and F. Harary, Planar permutation graphs, \emph{Ann. Inst. H.
Poincare,}  \textbf{3}(1967), 433-438.

\bibitem{yi} A. Chen, D. Ferrero, R. Gera and E. Yi, Functigraphs: An extension of permutation
graphs, \emph{Math. Bohem.,} \textbf{136}(1)(2011), 27-37.

\bibitem{col} C. J. Colbourn, P. J. Slater and  L. K. Stewart, Locating-dominating
sets in series parallel networks, \emph{Congr. Numer.},
\textbf{56}(1987), 135-162.

\bibitem{dor} W. Dorfler, On mapping graphs and permutation graphs, \emph{Math. Solvaca},
\textbf{28}(3)(1978), 277-288.

\bibitem{Linda1} L. Eroh, R. Gera, C. X. Kang, C. E. Larson and E. Yi,
Domination in functigraphs, \emph{arXiv preprint arXiv:1106.1147}.

\bibitem{Linda2} L. Eroh, C. X. Kang and E. Yi, On metric dimension of
functigraphs, \emph{Discrete Math., Alg. and Appl.}, \textbf{5(04)}
(2013), 1250060.

\bibitem{fazilfixing}
M. Fazil, I. Javaid and M. Murtaza, On fixing number of
functigraphs, \emph{arXiv preprint arXiv:1611.03346}.

\bibitem{fazildist}
M. Fazil, M. Mutaza, U. Ali and I. Javaid, On distinguishing number
of functigraphs, \emph{arXiv preprint arXiv:1612.00971}.


\bibitem{fin} A. Finbow and B. L. Hartnell, On locating-dominating
sets and well-covered graphs, \emph{Congr. Numer.},
\textbf{56}(1987), 135-162.

\bibitem{GuQu}
W. Gu and G. Qi, Attainability of the chromatic number of
functigraphs, \emph{In Pervasive Sys., Alg. and Net. (ISPAN), 2012
12th International Symposium}, (2012), 143-148.

\bibitem{hon} I. Honkala, T. Laihonen and S. Ranto, On locating-dominating codes in
binary hamming spaces, \emph{Disc. Math. Theor. Comput. Sci.},
\textbf{6}(2004), 265-282.

\bibitem{kang} C. X. Kang and E. Yi, On zero forcing number of
functigraphs, \emph{arXiv preprint arXiv:1204.2238}.

\bibitem{Karpov}
M. G. Karpovsky, K. Chakrabarty and L. B. Levitin, On a new class of
codes for identifying vertices in graphs, \emph{IEEE Transactions on
Information Theory} \textbf{44}(1998), 599-611.

\bibitem{Murtaza}
M. Murtaza, I. Javaid and M. Fazil, Locating-dominating sets and
identifying codes of a graph associated to a finite vector space,
\emph{arXiv preprint arXiv:1701.08537}.

\bibitem{QiWang}
G. Qi, S. Wang and W. Gu, On the chromatic number of functigraphs,
\emph{J. of Interconn. Net.}, \textbf{13}(2012), 1250011.



\bibitem{rall} D. F. Rall and P. J. Slater, On location-domination
numbers for certian classes of graphs, \emph{Congr. Numer.},
\textbf{45}(1984), 97-106.

\bibitem{slater1} P. J. Slater, Dominating and reference sets in a graph, \emph{J. Math.
Phys. Sci.}, \textbf{22}(1988), 445-455.

\bibitem{slater2} P. J. Slater, Fault-tolerant locating-dominating sets, \emph{Disc. Math.}, \textbf{249}(2002), 179-189.

\bibitem{slater3} P. J. Slater, Domination and location in acyclic graphs, \emph{Networks}, \textbf{17}(1987), 55-64.

\end{thebibliography}
\end{document}